\documentclass[12pt, oneside]{amsart}
\usepackage{amsmath}
\usepackage{amsfonts} 
\usepackage{amsthm, upref}
\usepackage{graphicx}
\usepackage[usenames, dvipsnames]{color}
\usepackage{bm}
\usepackage{soul}
\usepackage{enumerate}
\usepackage[head=1cm,headsep=.5cm,textheight=650pt,textwidth=16.5cm]{geometry}

\newcommand{\comment}[1]{}
\newcommand{\eq}{\begin{equation}}
\newcommand{\en}{\end{equation}}

\newcommand{\rr}{\mathbb{R}}
\newcommand{\nn}{\mathbb{N}}

\newcommand{\ev}{\mathbb E}

\newcommand{\ep}{\hfill $\Box$}

\numberwithin{equation}{section}

\begin{document}

\theoremstyle{plain}
\newtheorem{thm}{Theorem}
\newtheorem{lemma}[thm]{Lemma}
\newtheorem{prop}[thm]{Proposition}
\newtheorem{cor}[thm]{Corollary}

\theoremstyle{definition}
\newtheorem{defn}{Definition}
\newtheorem{asmp}{Assumption}
\newtheorem{notn}{Notation}
\newtheorem{prb}{Problem}

\theoremstyle{remark}
\newtheorem{rmk}{Remark}
\newtheorem{exm}{Example}
\newtheorem{clm}{Claim}

\title[Interacting Bessel processes]{Large deviations for interacting Bessel-like processes and applications to systemic risk}

\author{Tomoyuki Ichiba \and Mykhaylo Shkolnikov }

\thanks{Department of Statistics and Applied Probability, University of California, Santa Barbara, CA 93106 
}

\thanks{Department of Statistics, University of California, Berkeley, CA 94270 
}



\date{March 12, 2013}

\begin{abstract} 
We establish a process level large deviation principle for systems of interacting Bessel-like diffusion processes. By establishing weak uniqueness for the limiting non-local SDE of McKean-Vlasov type, we conclude that the latter describes the process level hydrodynamic limit of such systems and obtain a propagation of chaos result. This is the first instance of results of this type in the context of interacting diffusion processes under explicit assumptions on the coefficients, where the diffusion coefficients are allowed to be both non-Lipschitz and degenerate. In the second part of the paper, we explain how systems of this type naturally arise in the study of stability of the interbank lending system and describe some financial implications of our results.
\end{abstract}

\maketitle

\section{Introduction} 

Our main objects of study are systems of interacting diffusion processes on $\rr_+=[0,\infty)$ given by
\eq\label{SDE1}
\mathrm{d}X_i(t)=b(X_i(t),\rho^n(t))\,\mathrm{d}t + \sigma(X_i(t))\,\mathrm{d}W_i(t), \quad i=1,2,\ldots,n, 
\en
where $\rho^n(t)=\frac{1}{n}\sum_{i=1}^n \delta_{X_i(t)}$ is the empirical measure of the system at time $t$ and $b$, $\sigma$ are functions with values in $\rr$, $\rr_+$, respectively. We will be interested in situations, where $\sigma(0)=0$ and $\sigma$ is non-Lipschitz around $0$. The main example we have in mind is $\sigma(x)=c\,\sqrt{x}$ (see the applications below), but our analysis applies to suitable perturbations of such functions as well.   

\medskip

To state our main results, we need the following set of notations. We write $M_1(\rr_+)$ for the space of probability measures on $\rr_+$ endowed with the topology of weak convergence of measures, and let $\tilde{M}_1(\rr_+)$ be the space of probability measure on $\rr_+$ with finite first moments endowed with the $1$-Wasserstein distance. The latter is defined by
\eq
d_1(\mu,\nu)=\inf_p \int_{(\rr_+)^2} |x-y|\,\mathrm{d}p(x,y),
\en
where the infimum is taken over all probability measures $p$ on $(\rr_+)^2$ with first marginal given by $\mu$ and the second marginal given by $\nu$. Next, we fix a $T>0$ (which will remain fixed throughout) and let ${\mathcal Y}=C([0,T],\tilde{M}_1(\rr_+))$ be the space of continuous functions from $[0,T]$ to $\tilde{M}_1(\rr_+)$ equipped with the topology of uniform convergence. We set ${\mathcal X}=\tilde{M}_1(C([0,T],\rr_+))$ for the set of stochastic processes with continuous paths taking nonnegative values such that all their one-dimensional marginals have finite first moments. For an element $\gamma\in{\mathcal X}$, we write $(\pi(\gamma))(t)=\gamma(t)\in\tilde{M}_1(\rr_+)$, $t\in[0,T]$ for the path of one-dimensional marginals of $\gamma$. Finally, we define the topology on ${\mathcal X}$ as the collection of preimages of open sets in ${\mathcal Y}$, which are open in the topology of weak convergence on ${\mathcal X}$ (with $C([0,T],\rr_+)$ being endowed with the topology of uniform convergence). In other words, a sequence $\gamma_n$, $n\in\nn$ in ${\mathcal X}$ converges to a limit $\gamma\in{\mathcal X}$ iff $\gamma_n\rightarrow\gamma$ weakly and $\pi(\gamma_n)\rightarrow\pi(\gamma)$ in ${\mathcal Y}$.

\medskip

Our main results (large deviation principle, law of large numbers and propagation of chaos) hold under the following assumption.

\begin{asmp}\label{LDP_asmp}
The function $b$ is Lipschitz on $\rr_+\times\tilde{M}_1(\rr_+)$ and it holds $b(0,\cdot)>0$; whereas the function $\sigma$ is of the form $\sigma(x)=\sqrt{x}\,g(x)$ with a continuous bounded function $g:\,\rr_+\rightarrow\rr_+$ taking strictly positive values and satisfies
\eq
|\sigma(x)-\sigma(y)|\leq\Theta(|x-y|),\quad x,y\in\rr_+,
\en
where $\Theta:\,\rr_+\rightarrow\rr_+$ is such that $\int_0^\epsilon \frac{1}{\Theta(a)^2}\,\mathrm{d}a=\infty$ for all $\epsilon>0$. Moreover, the sequence $\rho^{n}(0)$, $n\in\nn$ is deterministic and converges in $\tilde{M}_1(\rr_+)$ to a limit $\lambda\in\tilde{M}_1(\rr_+)$, as $n \to +\infty$. 
\end{asmp}

The following are the three main results of our paper. 

\begin{thm}[Large deviation principle]\label{LDP_thm}
Under Assumption \ref{LDP_asmp}, the sequence $\rho^n$, $n\in\nn$ satisfies a large deviations principle on ${\mathcal X}$ with scale $n$ and good rate function 
\eq
J(\gamma)=\frac{1}{2}\inf_{(u,W)\in{\mathcal B}_\gamma} \ev\left[\int_0^T u(t)^2\,\mathrm{d}t\right],
\en
where ${\mathcal B}_\gamma$ is the set of $(u,W)$ such that $W$ is a standard Brownian motion, $u$ is a progressively measurable process with respect to the filtration generated by $W$ satisfying $\ev\Big[\int_0^T u(t)^2\,\mathrm{d}t\Big]<\infty$ and there is a weak solution of
\eq\label{contrlimitSDE}
\mathrm{d}\bar{X}(t)= \big( b(\bar{X}(t),{\mathcal L}(\bar{X}(t)))+u(t)\sigma(\bar{X}(t))\big) \,\mathrm{d}t
+\sigma(\bar{X}(t))\,\mathrm{d}W(t)
\en 
with law $\gamma$. Here, ${\mathcal L}(\bar{X}(t))$ is the law of $\bar{X}(t)$ and we use the convention that any infimum over an empty set is given by $+\infty$. 
\end{thm}

As we will see below, Theorem \ref{LDP_thm} implies the following law of large numbers and propagation of chaos result.  

\begin{cor}[Law of large numbers]\label{LLN_cor}
Under Assumption \ref{LDP_asmp}, the ${\mathcal X}$-valued sequence $\rho^n$, $n\in\nn$ converges in distribution to the law of the unique strong solution of the nonlocal SDE
\eq\label{limitSDE}
\mathrm{d}X(t)=b(X(t),{\mathcal L}(X(t)))\,\mathrm{d}t+\sigma(X(t))\,\mathrm{d}W(t),
\en
where ${\mathcal L}(X(t))$ is the law of $X(t)$ and $W$ is a standard Brownian motion.
\end{cor}

\begin{cor}[Propagation of chaos]\label{Poc_cor}
Suppose that Assumption \ref{LDP_asmp} holds and that for some $k\in\nn$ the initial conditions $(X_1(0),X_2(0),\ldots,X_k(0))$ converge to a limit $(\hat{x}_1,\hat{x}_2,\ldots,\hat{x}_k)$. Then, the law of the process $(X_1,X_2,\ldots,X_k)$ converges to the product measure $\mu^{\hat{x}_1}\otimes\mu^{\hat{x}_2}\otimes\cdots\otimes\mu^{\hat{x}_k}$, where $\mu^{\hat{x}_i}$, $i=1,2,\ldots,k$ are the laws of strong solutions to  
\eq\label{PocSDE}
\mathrm{d}\hat{X}_i(t)=b(\hat{X}_i(t),\rho^\infty(t))\,\mathrm{d}t+\sigma(\hat{X}_i(t))\,\mathrm{d}W_i(t),
\en
started from $\hat{x}_i$, $i=1,2,\ldots,k$, respectively, and $\rho^\infty$ is the law of the solution to \eqref{limitSDE} with initial condition distributed according to $\lambda$. 
\end{cor}

Interacting diffusions, for which the drift and diffusion coefficients of each process depend only the current value of that process and the current value of the empirical measure of the whole system, are usually referred to as diffusions interacting through their mean-field. The behavior of such processes has been a classical object of study in the statistical physics literature (see \cite{Ga} and the references there) following the seminal work \cite{McK}. However, large deviation results for this kind of interacting particle systems are rarely found. Notable exceptions are the articles \cite{DG}, \cite{DSVZ} and \cite{BDF}. 

In \cite{DG}, a large deviation principle is established for interacting diffusion processes as in \eqref{SDE1} under the assumptions of non-degeneracy of the diffusion coefficients and the well-posedness of the corresponding martingale problem. In \cite{DSVZ}, a large deviation principle is proven for systems of diffusions interacting both through the drift and through the diffusion coefficients, however for a very special type of interaction known as interaction through the ranks; there, also the non-degeneracy of the diffusion coefficients is one of the assumptions. Finally, in \cite{BDF} the authors give a large deviation principle for a very general class of diffusions interacting through their mean-field, but under non-explicit assumptions such as the strong existence and uniqueness of solutions to \eqref{SDE1}, weak uniqueness of solutions to \eqref{contrlimitSDE} and tightness of empirical measures corresponding to tilted versions of \eqref{SDE1} (the latter is equivalent to exponential tightness of the original particle systems). The latter three assumptions are difficult to check in general if the drift and diffusion coefficients are non-Lipschitz or unbounded. In our setting, the diffusion coefficient can be non-Lipschitz and unbounded. Nonetheless, we are able to show that the suitably modified abstract assumptions of \cite{BDF} hold in our setting and obtain Theorem \ref{LDP_thm} by adapting the proof in \cite{BDF} correspondingly. Another main contribution of the present paper is the proof of strong uniqueness for the non-local SDE \eqref{limitSDE}, which allows us to deduce Corollaries \ref{LLN_cor} and \ref{Poc_cor} from Theorem \ref{LDP_thm}.  
  
\medskip    

In \cite{Fouque2012} interacting diffusion processes have been proposed as models for the interbank lending system with the goal of analyzing the stability or instability of the latter. Following this program, we study the features of the weak solution of \eqref{SDE1} viewed as a model of the interbank lending system in the second part of the paper. In this context, $X_i(t)$ stands for the monetary reserve of bank $i$ at time $t$. Banks lend and borrow money from each other in order to hold the necessary amounts of monetary reserves. Considering appropriate choices of the functions $b$, $\sigma$, we illustrate the effect of the monetary flows, which are represented by the interactions through the drift coefficients. 

\medskip

As an application of Theorem \ref{LDP_thm} and Corollary \ref{LLN_cor}, we shall derive the dynamics of the limiting distribution of banks' monetary reserves as $n \to \infty$. We are interested in the effects of banks' interactions on (1) the long term behavior of the large interbank lending system, (2) the volatility of the limiting system and (3) the boundary behavior of the limiting system when most banks are almost bankrupt. We expect that such probabilistic studies of the interbank lending system will help understand the key features of systemic risk and some of the reasons for financial crises.   

\medskip

The rest of the paper is organized as follows. In section 2.1 we prove weak uniqueness for the limiting non-local SDE \eqref{limitSDE} (despite the fact that the diffusion coefficients in \eqref{SDE1} can be non-Lipschitz and degenerate) and thereby derive Corollary \ref{LLN_cor} from Theorem \ref{LDP_thm}. Subsequently, we show how Corollary \ref{Poc_cor} can be obtained from Corollary \ref{LLN_cor}. In section 2.2 we prove that the suitably modified abstract assumptions of \cite{BDF} apply in our setting and use the correspondingly modified arguments from \cite{BDF} to complete the proof of Theorem \ref{LDP_thm}. The main challenge to overcome here is the unboundedness of the coefficients in \eqref{SDE1}. In section 3 we apply these results to model the large interbank lending system by choosing a particular form of the functions $b$, $\sigma$. The solution of the non-local SDE \eqref{limitSDE} then becomes a non-local squared radial Ornstein-Uhlenbeck process (see Corollary \ref{TheCor}). In section 3.1 we discuss the moments of the latter, recurrence and transience (see Propositions \ref{Xboundary}-\ref{Xmoments}), whereas in section 3.2 we study the effect of the interactions on the volatility of the limiting system. Section 3.3 is devoted to the derivation of a quasilinear first-order PDE satisfied by the Laplace transforms of one-dimensional distributions of non-local squared radial Ornstein-Uhlenbeck processes. Lastly, in section 3.4 we compute the stationary distribution of the latter, which corresponds to the long-term behavior of the limiting interbank lending system.   

\section{Large deviation principle}

\subsection{Proof of Corollary \ref{LLN_cor}}

The proof of Corollary \ref{LLN_cor} relies on the following uniqueness statement for the limiting non-local SDE, which is of independent interest. 

\begin{prop}\label{uniq_prop}
Let $b$, $\sigma$ be as in Assumption \ref{LDP_asmp}. Then, the solution of the non-local stochastic differential equation
\eq\label{limitSDE2}
\mathrm{d}X(t)=b(X(t),{\mathcal L}(X(t)))\,\mathrm{d}t+\sigma(X(t))\,\mathrm{d}W(t)
\en
is pathwise unique. In particular, the law of the pair $(X,W)$ is uniquely determined. 
\end{prop}

\begin{proof}
Let $X$, $\tilde{X}$ be two solutions of \eqref{limitSDE2} with respect to the same Brownian motion $W$ and with the same initial value. Define the stopping times
\eq
\tau_L:=\inf\{t\geq0:\;\max(X(t),\tilde{X}(t))\geq L\}, \quad L\in\nn. 
\en
In view of Corollary IX.3.4 in \cite{RY}, the local time of $X-\tilde{X}$ at zero is identically equal to zero. Therefore, Tanaka's formula implies that the process
\eq
|X(\cdot\wedge\tau_L)-\tilde{X}(\cdot\wedge\tau_L)|
-\int_0^{\cdot\wedge\tau_L} \mathrm{sgn}(X(s)-\tilde{X}(s))\big(b(X(s),{\mathcal L}(X(s)))-b(\tilde{X}(s),{\mathcal L}(\tilde{X}(s)))\big)\,\mathrm{d}s
\en 
is a martingale starting at zero, where $\mathrm{sgn}(x)=\mathbf{1}_{\{x>0\}} - \mathbf{1}_{\{x \le 0\}}$. Evaluating at a fixed $t>0$, taking the expectation and using the fact that the function $b$ is Lipschitz (see Assumption \ref{LDP_asmp}), we obtain
\[
\ev[|X(t\wedge\tau_L)-\tilde{X}(t\wedge\tau_L)|]
\leq C\,\ev\Big[\int_0^{t\wedge\tau_L} \big[ |X(s)-\tilde{X}(s)|+d_1({\mathcal L}(X(s)),{\mathcal L}(\tilde{X}(s)))\big]\,\mathrm{d}s\Big] 
\] 
with a uniform constant $C<\infty$. Moreover, the definition of the $1$-Wasserstein distance implies
\[
d_1({\mathcal L}(X(s)),{\mathcal L}(\tilde{X}(s)))\leq \ev[|X(s)-\tilde{X}(s)|], \quad s\geq0. 
\]
Passing to the limit $L\rightarrow\infty$, using Fatou's Lemma for the left-hand side and the Monotone Convergence Theorem for the right-hand side, we obtain
\[
\ev[|X(t)-\tilde{X}(t)|]\leq 2\,C\,\int_0^t \ev[|X(s)-\tilde{X}(s)|]\,\mathrm{d}s.
\]
Therefore, Gronwall's Lemma shows that the function $t\mapsto\ev[|X(t)-\tilde{X}(t)|]$ vanishes identically. This yields the desired pathwise uniqueness. Finally, in view of the theory of Yamada and Watanabe (see Proposition 5.3.20 in \cite{MR1121940}), pathwise uniqueness implies that the law of the pair $(X,W)$ is uniquely determined.
\end{proof}

We are now ready to give the proof of Corollary \ref{LLN_cor}, taking Theorem \ref{LDP_thm} for granted.

\medskip

\noindent\textit{Proof of Corollary \ref{LLN_cor}.} The formula for the rate function in Theorem \ref{LDP_thm} shows that $J(\gamma)=0$ holds if and only if $\gamma$ is the law of a solution to the nonlocal SDE \eqref{limitSDE2}. By Proposition \ref{uniq_prop} this law is unique, and Corollary \ref{LLN_cor} now follows from the goodness of the rate function. \ep

\medskip

Next, we show that Corollary \ref{Poc_cor} is a consequence of Corollary \ref{LLN_cor}.

\medskip

\noindent\textit{Proof of Corollary \ref{Poc_cor}.} We fix $k$ and $\hat{x}_1,\hat{x}_2,\ldots,\hat{x}_k$ as in the statement of Corollary \ref{Poc_cor}. Corollary \ref{LLN_cor} implies that the sequence of random variables $\rho^n$, $n\in\nn$ is tight on $M_1(C([0,T],\rr_+))$. Therefore, by Proposition 2.2 in \cite{Sz} the sequence of processes $(X_1,X_2,\ldots,X_k)$ as $n$ varies is tight as well. Moreover, the same argument as in the proof of Corollary \ref{LLN_cor} shows that the solution of the SDE \eqref{PocSDE} is pathwise unique. Therefore, it suffices to identify the limit points of the sequence of $(X_1,X_2,\ldots,X_k)$ as $n$ varies as the unique solution of the following martingale problem: for each $f\in C^\infty_c(\rr^k)$ the process
\eq\label{Mart}
f(\hat{X}(\cdot))-f(\hat{x})-\sum_{i=1}^k \int_0^\cdot b(\hat{X}_i(t),\rho^\infty(t))\,\frac{\partial f}{\partial x_i}(\hat{X}(t)) 
+ \frac{1}{2}\sigma(\hat{X}_i(t))^2\,\frac{\partial^2 f}{\partial x_i^2}(\hat{X}(t))\,\mathrm{d}t
\en
is a martingale, with $\hat{x}=(\hat{x}_1,\hat{x}_2,\ldots,\hat{x}_k)$. To this end, it suffices to note that by Corollary \ref{LLN_cor} each process as in \eqref{Mart} is the limit of the processes 
\begin{eqnarray*}
M^{n,f}(\cdot):=f(X_1(\cdot),\ldots,X_k(\cdot))-f(X_1(0),\ldots,X_k(0))\\
-\sum_{i=1}^k \int_0^\cdot b(X_i(t),\rho^n(t))\,\frac{\partial f}{\partial x_i}(X_1(t),\ldots,X_k(t))
+ \frac{1}{2}\sigma(X_i(t))^2\,\frac{\partial^2 f}{\partial x_i^2}(X_1(t),\ldots,X_k(t))\,\mathrm{d}t \\
= \sum_{i=1}^k \int_0^\cdot \sigma(X_i(t))\frac{\partial f}{\partial x_i}(X_1(t),\ldots,X_k(t))\,\mathrm{d}W_i(t)
\end{eqnarray*}
as $n\rightarrow\infty$, and that sequence of random variables $M^{n,f}(t)$, $n\in\nn$ is uniformly integrable for any fixed $t\in[0,T]$, since the second moments of the latter are uniformly bounded by virtue of It\^o's isometry, the continuity of $\sigma$ and the fact that $f$ has compact support. \ep 

\subsection{Proof of Theorem \ref{LDP_thm}}

This subsection is devoted to the proof of Theorem \ref{LDP_thm}. We first recall the framework of \cite{BDF} as our approach will rely on a version of the main result there. 

\medskip

The large deviations principle in \cite{BDF} is established in form of a Laplace principle, with the local large deviations upper and lower bounds obtained by considering suitable \textit{tilted} versions of the original system of SDEs. In our case, the appropriate controlled version of the system \eqref{SDE1} is given by 
\eq\label{control_part}
\mathrm{d}\bar{X}_i(t)=\big(b(\bar{X}_i(t),\bar{\rho}^n(t))+u_i(t)\sigma(\bar{X}_i(t))\big)\,\mathrm{d}t + \sigma(\bar{X}_i(t))\,\mathrm{d}W_i(t),\quad i=1,2,\ldots,n.
\en 
Hereby, the control processes $u_1,u_2,\ldots,u_n$ are assumed to be progressively measurable with respect to the filtration generated by the Brownian motions $W_1,W_2,\ldots,W_n$ and to satisfy 
\eq
\sum_{i=1}^n \, \ev\Big[\int_0^T u_i(t)^2\,\mathrm{d}t\Big] < \infty,
\en
and $\bar{\rho}^n(t)=\frac{1}{n}\sum_{i=1}^n \delta_{\bar{X}_i(t)}$ is the empirical measure of the controlled system at time $t$. 

\medskip

In addition, we introduce a controlled version of the limiting non-local SDE \eqref{limitSDE}:
\eq\label{control_limit}
\mathrm{d}\bar{X}(t)=\big(b(\bar{X}(t),{\mathcal L}(\bar{X}(t)))+u(t)\sigma(\bar{X}(t))\big)\mathrm{d}t + \sigma(\bar{X}(t))\,\mathrm{d}W(t),
\en
where $u$ is a progressively measurable process with respect to the filtration generated by the Brownian motion $W$, which satisfies
\eq
\ev\left[\int_0^T u(t)^2\,\mathrm{d}t\right]<\infty,
\en
and ${\mathcal L}(\bar{X}(t))$ is the law of $\bar{X}(t)$. 

\medskip

The exponential tightness of the original sequence $\rho^n$, $n\in\nn$ of empirical measures corresponds in this setup to the precompactness of the sets $\{\bar{\rho}^n:\;n\in\nn\}$ for all sequences of controls  satisfying
\eq\label{contr_bound}
\sup_{n\in\nn}\,\frac{1}{n}\sum_{i=1}^n\,\ev\left[\int_0^T |u_i(t)|^2\,\mathrm{d}t\right] < \infty.
\en

\smallskip

In this situation, a slightly modified version of the main result in \cite{BDF} reads as follows.

\begin{prop}\label{BDF_prop}
Assume the following. 
\begin{enumerate}[(i)]
\item The initial conditions $\rho^n(0)$, $n\in\nn$ are deterministic and converge in $\tilde{M}_1(\rr_+)$ to a limit $\lambda$.
\item The coefficients $b$, $\sigma$ are continuous on $\rr_+\times\tilde{M}_1(\rr)$, $\rr_+$, respectively. 
\item Strong existence and uniqueness holds for \eqref{SDE1}.
\item Weak uniqueness in ${\mathcal X}$ holds for \eqref{control_limit} for any given joint distribution of $(u,W)$. 
\item For any sequence of controls satisfying \eqref{contr_bound}, the set $\{\bar{\rho}^n:\;n\in\nn\}$ is precompact in ${\mathcal X}$. 
\end{enumerate}
Then, the sequence $\rho^n$, $n\in\nn$ satisfies a large deviations principle on ${\mathcal X}$ with scale $n$ and a good rate function
\eq
J(\gamma)=\frac{1}{2}\inf_{(u,W)\in{\mathcal B}_\gamma} \ev\left[\int_0^T u(t)^2\,\mathrm{d}t\right],
\en
where ${\mathcal B}_\gamma$ is the set of all $(u,W)$ such that the corresponding weak solution $\bar{X}$ of \eqref{control_limit} has law $\gamma$. 
\end{prop} 

\begin{proof}
The proposition is obtained by carrying out the proof of Theorem 3.1 in \cite{BDF} replacing their space ${\mathcal X}$ by our space ${\mathcal X}$. 
\end{proof}

One of the main contributions of our paper is to show that the abstract assumptions (iii)-(v) in Proposition \ref{BDF_prop} are satisfied under the concrete Assumption \ref{LDP_asmp} on the coefficients in \eqref{SDE1}, so that Theorem \ref{LDP_thm} becomes a consequence of Proposition \ref{BDF_prop}.    

\medskip

\noindent\textit{Proof of Theorem \ref{LDP_thm}.} In view of Proposition \ref{BDF_prop}, it suffices to check that the assumptions (i)-(v) there are consequences of Assumption \ref{LDP_asmp}. The assertions (i) and (ii) are immediate, so that the rest of the proof is devoted to the derivation of assertions (iii), (iv) and (v), which are dealt with in steps 1, 2 and 3, respectively.  

\medskip

\noindent\textit{Step 1.} To prove (iii), we note first that a weak solution of \eqref{SDE1} exists by virtue of Theorem A in \cite{BP}. Therefore, it suffices to prove the pathwise uniqueness for the solutions of \eqref{SDE1}. To this end, we let $(X_1,X_2,\ldots,X_n)$ and $(\tilde{X}_1,\tilde{X}_2,\ldots,\tilde{X}_n)$ be two solutions of \eqref{SDE1} on the same probability space and with respect to the same Brownian motions $W_1,W_2,\ldots,W_n$. Then, by Corollary IX.3.4 in \cite{RY}, the local times of $X_1-\tilde{X}_1,X_2-\tilde{X}_2,\ldots,X_n-\tilde{X}_n$ at zero all vanish identically. Therefore, for any fixed $i\in\{1,2,\ldots,n\}$ and $L\in\nn$, the process 
\[
|X_i(\cdot\wedge\tau^{(n)}_L)-\tilde{X}_i(\cdot\wedge\tau^{(n)}_L)|
-\int_0^{\cdot\wedge\tau^{(n)}_L} \mathrm{sgn}(X_i(s)-\tilde{X}_i(s))\big(b(X_i(s),\rho^n(s))-b(\tilde{X}_i(s),\tilde{\rho}^n(s))\big)\,\mathrm{d}s
\]   
is a martingale starting at zero, where 
\[
\tau^{(n)}_L=\inf\{t\geq0:\;\max_{i=1,2,\ldots,n} \max(X_i(t),\tilde{X}_i(t))\geq L\}.
\]
Since the function $b$ is assumed to be Lipschitz (see Assumption \ref{LDP_asmp}) and one has
\[
d_1\Big(\frac{1}{n}\sum_{j=1}^n \delta_{x_j},\frac{1}{n}\sum_{j=1}^n \delta_{\tilde{x}_j}\Big)\leq \frac{1}{n}\sum_{j=1}^n |x_j-\tilde{x}_j|,
\]
it follows 
\eq
\sum_{i=1}^n \ev[|X_i(\cdot\wedge\tau^{(n)}_L)-\tilde{X}_i(\cdot\wedge\tau^{(n)}_L)|]
\leq C\,\int_0^\cdot \sum_{i=1}^n \ev[|X_i(s\wedge\tau^{(n)}_L)-\tilde{X}_i(s\wedge\tau^{(n)}_L)|]\,\mathrm{d}s
\en
with a uniform constant $C<\infty$. Applying Gronwall's Lemma first and then taking the limit $L\rightarrow\infty$, we conclude $X_i(t)=\tilde{X}_i(t)$, $i=1,2,\ldots,n$ for any fixed $t\in[0,T]$ with probability $1$. Thus, due to the path continuity of both solutions it must hold $X_i(\cdot)=\tilde{X}_i(\cdot)$, $i=1,2,\ldots,n$ with probability $1$. 

\medskip

\noindent\textit{Step 2.} We now turn to the proof of assertion (iv) in Proposition \ref{BDF_prop}. By the theory of Yamada and Watanabe (see Proposition 5.3.20 of \cite{MR1121940}) it suffices to show that the solution of \eqref{control_limit} is pathwise unique. To this end, let $\bar{X}^1$, $\bar{X}^2$ be two solutions of \eqref{control_limit} on the same probability space and with respect to the same pair $(u,W)$. Now, introduce the stopping times
\eq
\sigma_M:=\inf\Big\{t\geq 0:\;\int_0^t u(s)^2\,\mathrm{d}s \geq M\Big\},\quad M\in\nn,
\en
and for each fixed $M\in\nn$ define a change of measure according to the density
\eq
\frac{\mathrm{d}P^M}{\mathrm{d}P}
=\exp\Big(-\int_0^{T\wedge\sigma_M} u(s)\,\mathrm{d}W(s)-\frac{1}{2} \int_0^{T\wedge\sigma_M} u(s)^2\,\mathrm{d}s\Big).
\en
In view of Novikov's condition, $P^M$ is a well-defined probability measure, which is equivalent to $P$. Moreover, by Girsanov's Theorem it holds
\begin{eqnarray}
\quad\quad\quad\bar{X}^1(\cdot\wedge\sigma_M)=\bar{X}^1(0)
+\int_0^{\cdot\wedge\sigma_M} b(\bar{X}^1(t),{\mathcal L}_P(\bar{X}^1(t)))\,\mathrm{d}t
+\int_0^{\cdot\wedge\sigma_M} \sigma(\bar{X}^1(t))\,\mathrm{d}W^M(t), \\
\quad\quad\quad\bar{X}^2(\cdot\wedge\sigma_M)=\bar{X}^2(0)
+\int_0^{\cdot\wedge\sigma_M} b(\bar{X}^2(t),{\mathcal L}_P(\bar{X}^2(t)))\,\mathrm{d}t
+\int_0^{\cdot\wedge\sigma_M} \sigma(\bar{X}^2(t))\,\mathrm{d}W^M(t)
\end{eqnarray} 
under $P^M$, where ${\mathcal L}_P(\bar{X}^k(t))$, $k=1,2$ are the laws of $\bar{X}^k(t)$, $k=1,2$ under $P$ and $W^M(\cdot\wedge\sigma_M)+\int_0^{\cdot\wedge\sigma_M} u(t)\,\mathrm{d}t$ is a standard Brownian motion stopped at $\sigma_M$ under $P^M$. Now, arguing as in the proof of Proposition \ref{uniq_prop}, we conclude that there exists a uniform constant $C<\infty$ such that
\eq\label{Mbound}
|\bar{X}^1(\cdot\wedge\sigma_M)-\bar{X}^2(\cdot\wedge\sigma_M)|
\leq C\,\int_0^{\cdot\wedge\sigma_M} |\bar{X}^1(t)-\bar{X}^2(t)|+d_1({\mathcal L}_P(\bar{X}^1(t)),{\mathcal L}_P(\bar{X}^2(t)))\,\mathrm{d}t
\en 
holds with probability $1$ under $P^M$. However, since for any fixed $M\in\nn$ the measures $P^M$ and $P$ are equivalent, \eqref{Mbound} must hold under $P$ with probability $1$ for all $M\in\nn$. Taking the expectation in \eqref{Mbound} under $P$ and taking the limit $M\rightarrow\infty$, we obtain 
\eq
\ev[|\bar{X}^1(\cdot)-\bar{X}^2(\cdot)|]
\leq C\,\int_0^\cdot \ev[|\bar{X}^1(t)-\bar{X}^2(t)|]+d_1({\mathcal L}_P(\bar{X}^1(t)),{\mathcal L}_P(\bar{X}^2(t)))\,\mathrm{d}t
\en
in view of Fatou's Lemma and the Monotone Convergence Theorem. Recalling that
\eq
d_1({\mathcal L}_P(\bar{X}^1(t)),{\mathcal L}_P(\bar{X}^2(t)))\leq \ev[|\bar{X}^1(t)-\bar{X}^2(t)|],
\en
applying Gronwall's Lemma and using the path continuity of $\bar{X}^1$ and $\bar{X}^2$, we end up with $\bar{X}^1(\cdot)=\bar{X}^2(\cdot)$ on $[0, T]$ with probability $1$ under $P$. 
 
\medskip

\noindent\textit{Step 3.} Lastly, we need to show the assertion (v) in Proposition \ref{BDF_prop}. By the definition of the topology on ${\mathcal X}$, the precompactness of the set $\{\bar{\rho}^n:\,n\in\nn\}$ amounts to showing that the sequence $\bar{\rho}^n$, $n\in\nn$ is tight with respect to the topology of weak convergence on $M_1(C([0,T],\rr_+))$ and that the set $\{\pi(\bar{\rho}^n):\,n\in\nn\}$ is precompact in ${\mathcal Y}=C([0,T],\tilde{M}_1(\rr_+))$. We prove these two assertions in steps 3A and 3B, respectively. 

\medskip

{\noindent\textit{Step 3A.}} We need to verify that for every $\epsilon>0$ there exists a precompact set $K^\epsilon\subset M_1(C([0,T],\rr_+))$ such that
\eq\label{tight1}
\sup_{n\in\nn} P(\bar{\rho}^n\notin K^\epsilon)\leq\epsilon. 
\en
Moreover, by Prokhorov's Theorem, it is sufficient to demonstrate that for each $\epsilon>0$ we can find a set $K^\epsilon$ satisfying \eqref{tight1} such that for every $L\in\nn$ there is a precompact set $K^{\epsilon,L}\subset C([0,T],\rr_+)$ with
\eq
\sup_{\gamma\in K^\epsilon} \gamma(C([0,T],\rr_+)\backslash K^{\epsilon,L})\leq 1/L. 
\en
To this end, for any function $f\in C([0,T],\rr_+)$ and any $\zeta\in(0,1]$, we let $c_\zeta(f)$ be the modulus of continuity of $f$ evaluated at $\zeta$, and write $\|.\|_\infty$ for the supremum-norm on $C([0,T],\rr_+)$. Next, we set
\begin{eqnarray}
&&K^{\epsilon,L}=\{f\in C([0,T],\rr_+):\;c_{1/L}(f)+\|f\|_\infty\leq M(\epsilon,L)\},\\
&&K^\epsilon=\{\gamma\in M_1(C([0,T],\rr_+)):\;\gamma(C([0,T],\rr_+)\backslash K^{\epsilon,L})\leq 1/L,\,L\in\nn\}
\end{eqnarray} 
with a large enough constant $M(\epsilon,L)<\infty$ to be chosen later. With these choices \eqref{tight1} reads as
\eq\label{tight2}
\sup_{n\in\nn} 
P\Big(\exists\,L\in\nn:\;\big|\{ 1\le i \le n :\,c_{1/L}(\bar{X}_i(\cdot))+\|\bar{X}_i(\cdot)\|_\infty>M(\epsilon,L)\}\big|>n/L\Big)\leq\epsilon,
\en
where $\bar{X}_1,\bar{X}_2,\ldots,\bar{X}_n$ is a weak solution of the controlled system \eqref{control_part}. Moreover, \eqref{tight2} would follow if we can show the stronger bound 
\eq\label{tight3}
\sup_{n\in\nn} \, \sum_{L=1}^\infty P\big(c_{1/L}(\bar{\mathfrak X}(\cdot))+\|\bar{\mathfrak X}(\cdot)\|_\infty>M(\epsilon,L)n/L\big)\leq\epsilon,
\en
where we wrote $\bar{\mathfrak X}(\cdot)=(\bar{X}_1(\cdot),\bar{X}_2(\cdot),\ldots,\bar{X}_n(\cdot))$, and used the notations $c_{1/L}$ and $\|\cdot\|_\infty$ for the modulus of continuity at $1/L$ and the supremum-norm, respectively, of a function in $C([0,T],(\rr_+)^n)$ with respect to the $1$-{\it norm} on $(\rr_+)^n$. 

\medskip

We proceed by obtaining an upper bound on $\ev[\|\bar{\mathfrak X}(\cdot)\|_\infty]$. To this end, we note that by the Cauchy-Schwarz inequality
\begin{eqnarray*}
\bar{X}_i(\cdot)
\leq \bar{X}_i(0)+\int_0^\cdot |b(\bar{X}_i(s),\bar{\rho}^n(s))|\,\mathrm{d}s
+\|g\|_\infty\,\Big(\int_0^\cdot \bar{X}_i(s)\,\mathrm{d}s\Big)^{1/2} \Big(\int_0^\cdot u_i(s)^2\,\mathrm{d}s\Big)^{1/2} \\
+\Big|\int_0^\cdot \sigma(\bar{X}_i(s))\,\mathrm{d}W_i(s)\Big|. 
\end{eqnarray*}
Taking the supremum over a time interval $[0,\cdot]$ and the expectation on both sides, applying the Cauchy-Schwarz inequality again and writing $r_i(\cdot):=\ev[\sup_{s\in[0,\cdot]} \bar{X}_i(s)]$, we get
\begin{eqnarray*}
r_i(\cdot)
\leq \bar{X}_i(0) + \ev\Big[\sup_{s\in[0,\cdot]}
\int_0^s |b(\bar{X}_i(0),\bar{\rho}^n(0))|+C|\bar{X}_i(a)-\bar{X}_i(0)|+C d_1(\bar{\rho}^n(a),\bar{\rho}^n(0))\,\mathrm{d}a \Big] \\
+\|g\|_\infty\,\Big(\int_0^\cdot r_i(s)\,\mathrm{d}s\Big)^{1/2} \kappa_i^{1/2} 
+ \ev\Big[\sup_{0\leq s\leq \cdot} \Big|\int_0^s \sigma(\bar{X}_i(a))\,\mathrm{d}W_i(a)\Big|\Big] 
\end{eqnarray*}
with a constant $C<\infty$ depending only on the Lipschitz constant of $b$ and 
\eq
\kappa_i=\ev\Big[\int_0^T u_i(t)^2\,\mathrm{d}t\Big],\quad i=1,2,\ldots,n. 
\en
Summing over $i$, applying the elementary inequality $2a_1a_2\leq a_1^2+a_2^2$, $a_1,a_2\in\rr$ and setting $r(\cdot)=\sum_{i=1}^n r_i(\cdot)$, we obtain
\begin{eqnarray*}
r(\cdot)\leq \sum_{i=1}^n \Big(\bar{X}_i(0)+|b(\bar{X}_i(0),\bar{\rho}^n(0))|T+\|g\|_\infty\kappa_i/2 
+\ev\Big[\sup_{0\leq s\leq\cdot} \Big|\int_0^s \sigma(\bar{X}_i(a))\,\mathrm{d}W_i(a)\Big|\Big]\Big)\\
+C\int_0^\cdot r(s)\,\mathrm{d}s, 
\end{eqnarray*} 
where $C<\infty$ now stands for a constant depending on the Lipschitz constant of $b$ and $\|g\|_\infty$ only. At this point, Gronwall's Lemma yields the estimate
\[
r(t)\leq \sum_{i=1}^n \Big(\bar{X}_i(0)+|b(\bar{X}_i(0),\bar{\rho}^n(0))|T+\|g\|_\infty\kappa_i/2 
+\ev\Big[\sup_{0\leq s\leq t} \Big|\int_0^s \sigma(\bar{X}_i(a))\,\mathrm{d}W_i(a)\Big|\Big]\Big)e^{Ct},
\] 
$t\in[0,T]$. Next, applying the $L^2$ version of Doob's maximal inequality for nonnegative submartingales and It\^o's isometry, we obtain
\[
r(t)\leq \sum_{i=1}^n \Big(\bar{X}_i(0)+|b(\bar{X}_i(0),\bar{\rho}^n(0))|T+\|g\|_\infty\kappa_i/2 
+\|g\|_\infty2\,\ev\Big[\int_0^t \bar{X}_i(s)\,\mathrm{d}s\Big]^{1/2}\Big)e^{Ct},\;t\in[0,T]. 
\]
The elementary inequality $2a\leq 1+a^2$, $a\in\rr$ now gives
\[
r(t)\leq \sum_{i=1}^n \Big(\bar{X}_i(0)+|b(\bar{X}_i(0),\bar{\rho}^n(0))|T+\|g\|_\infty\kappa_i/2+\|g\|_\infty\Big) e^{Ct}
+\|g\|_\infty \int_0^t r(s)\,\mathrm{d}s\,e^{Ct},\;t\in[0,T].
\]  
Applying Gronwall's Lemma once again and recalling \eqref{contr_bound}, we conclude that there exists a constant $\tilde{C}<\infty$ depending only on $b$, $g$, $T$ and the supremum in \eqref{contr_bound} such that
\eq\label{normbound}
\ev[\|\bar{\mathfrak X}(\cdot)\|_\infty]=r(T)\leq \tilde{C}n.
\en

\medskip

Next, we give an upper bound on $\ev[c_{1/L}(\bar{\mathfrak X}(\cdot))]$. To this end, we note first that
\begin{eqnarray*}
\sum_{i=1}^n |\bar{X}_i(t)-\bar{X}_i(s)|
\leq \sum_{i=1}^n \int_s^t |b(\bar{X}_i(a),\bar{\rho}^n(a))|\,\mathrm{d}a
+ \frac{1}{2}\sum_{i=1}^n\int_s^t \big[ g(\bar{X}_i(a))^2\bar{X}_i(a)+u_i(a)^2\big]\,\mathrm{d}a\\
+\sum_{i=1}^n \Big|\int_s^t \sigma(\bar{X}_i(a))\,\mathrm{d}W_i(a)\Big|,\quad 0\leq s<t\leq T,
\end{eqnarray*}
where we have used the elementary inequality $2a_1a_2\leq a_1^2+a_2^2$, $a_1,a_2\in\rr$. Taking the supremum over all $0\leq s<t\leq T$ with $t-s\leq 1/L$ and then the expected value on both sides and using the fact that $b$ is Lipschitz (see Assumption \ref{LDP_asmp}), we get
\begin{eqnarray*}
\ev[c_{1/L}(\bar{\mathfrak X}(\cdot))]\leq \frac{1}{L} \sum_{i=1}^n |b(\bar{X}_i(0),\bar{\rho}^n(0))| + \frac{C}{L}\,r(T)
+\frac{1}{2}\sum_{i=1}^n \kappa_i \quad\quad\quad\quad\quad\quad\quad\quad\quad\\
+2\sum_{i=1}^n\ev\Big[\sup_{k=0,\ldots,[LT]} \sup_{t\in[k/L,(k+1)/L]} \Big|\int_{k/L}^t \sigma(\bar{X}_i(a))\,\mathrm{d}W_i(a)\Big|\Big]
\end{eqnarray*}
with a constant $C<\infty$ depending only on the Lipschitz constant of $b$ and $\|g\|_\infty$. Next, applying the $L^2$-version of Doob's maximal inequality for nonnegative submartingales, we can bound $\ev[c_{1/L}(\bar{\mathfrak X}(\cdot))]$ further by
\begin{eqnarray*}
&& \frac{1}{L}\,\sum_{i=1}^n |b(\bar{X}_i(0),\bar{\rho}^n(0))| + \frac{C}{L}\,r(T)
+\frac{1}{2}\,\sum_{i=1}^n \kappa_i + \frac{([LT]+1)\,C}{\sqrt{L}}\,\sum_{i=1}^n r_i(T)^{1/2}\\
&\leq& \frac{1}{L}\,\sum_{i=1}^n |b(\bar{X}_i(0),\bar{\rho}^n(0))| + \frac{C}{L}\,r(T)
+\frac{1}{2}\,\sum_{i=1}^n \kappa_i + \frac{([LT]+1)\,C}{\sqrt{L}}(n+r(T)),
\end{eqnarray*}
where we have increased the value of $C$ if necessary. In view of \eqref{contr_bound} and \eqref{normbound}, this implies the existence of a constant $\hat{C}<\infty$ depending only on $b$, $g$, $T$, $L$ and the supremum in \eqref{contr_bound} such that
\eq\label{modofcontbound}
\ev[c_{1/L}(\bar{\mathfrak X}(\cdot))]\leq\hat{C}n. 
\en
At this point, the bound \eqref{tight3} with $M(\epsilon,L)$ large enough is a direct consequence of Markov's inequality, \eqref{normbound} and \eqref{modofcontbound}.

\medskip

\noindent\textit{Step 3B.} It remains to check the precompactness of the set $\{\pi(\bar{\rho}^n):\;n\in\nn\}$ in $\mathcal{Y}$. We first show the precompactness of the latter set in $C([0,T],M_1(\rr_+))$ with the topology of uniform convergence, where $M_1(\rr_+)$ is endowed with a metric metrizing the topology of weak convergence of probability measures. Letting
\eq\label{whatisF}
{\mathcal F}=\Big\{f\in C^\infty(\rr_+):\;\limsup_{x\rightarrow\infty} \frac{|f(x)|}{x}<\infty,\;\|f'\|_\infty<\infty,\;\|f''\|_\infty<\infty\Big\},
\en
and arguing as in the proof of Lemma 1.3 in \cite{Ga}, we conclude that precompactness in $C([0,T],M_1(\rr_+))$ would follow if we can show the tightness of the sequence 
\eq
(\bar{\rho}^n(\cdot),f):=\int_{\rr_+} f\,\mathrm{d}\bar{\rho}^n(\cdot),\quad n\in\nn
\en 
on $C([0,T],\rr)$ for every fixed $f\in{\mathcal F}$. In order to prove the latter, we aim to apply the criterion in Theorem 1.3.2 of \cite{MR2190038} and therefore need to show that
\eq
\forall\,f\in{\mathcal F},\,\Delta>0:\quad \lim_{\epsilon\downarrow0}\,\limsup_{n\rightarrow\infty}\, 
P\Big(\sup_{0\leq s<t\leq T,t-s<\epsilon} |(\bar{\rho}^n(t),f)-(\bar{\rho}^n(s),f)|>\Delta\Big)=0.
\en   
Applying It\^o's formula, using the union bound and recalling that $b$ is Lipschitz and $f'$, $f''$ are bounded for any fixed $f\in{\mathcal F}$, we conclude that it is enough to verify
\begin{eqnarray*}
&&\forall\,\Delta_1>0:\quad \lim_{\epsilon\downarrow0} \limsup_{n\rightarrow\infty} 
P\Big(\frac{1}{n}\sum_{i=1}^n \sup_{0\leq s<t\leq T,t-s<\epsilon} \int_s^t \bar{X}_i(a)\,\mathrm{d}a>\Delta_1\Big)=0,\\
&&\forall\,\Delta_2>0:\quad \lim_{\epsilon\downarrow0} \limsup_{n\rightarrow\infty} 
P\Big(\frac{1}{n}\sum_{i=1}^n \sup_{0\leq s<t\leq T,t-s<\epsilon} \int_s^t \sigma(\bar{X}_i(a))\,|u_i(a)|\,\mathrm{d}a>\Delta_2\Big)=0,\\
&&\forall\,\Delta_3>0:\quad \lim_{\epsilon\downarrow0} \limsup_{n\rightarrow\infty} 
P\Big(\sup_{0\leq s<t\leq T,t-s<\epsilon} \Big|\frac{1}{n}\sum_{i=1}^n \int_s^t f'(\bar{X}_i(a))\sigma(\bar{X}_i(a))\,\mathrm{d}W_i(a)\Big|>\Delta_3\Big)=0. 
\end{eqnarray*}

To prove the first assertion, we apply Markov's inequality in combination with \eqref{normbound} to bound the prelimit expression from above by $\frac{\tilde{C}\epsilon}{\Delta_1}$. Taking limits, one ends up with the first assertion. 

\medskip

To show the second assertion, we first use Markov's inequality and then repeatedly the Cauchy-Schwarz inequality to bound the prelimit expression from above by
\begin{eqnarray*}
&&\frac{C}{n\Delta_2}\,\sum_{i=1}^n \ev\Big[\sup_{0\leq s<t\leq T,t-s<\epsilon} 
\Big(\int_s^t \bar{X}_i(a)\,\mathrm{d}a\Big)^{1/2} \Big(\int_s^t u_i(a)^2\,\mathrm{d}a\Big)^{1/2} \Big] \\
&\leq& \frac{C\sqrt{\epsilon}}{n\Delta_2}\,\sum_{i=1}^n r_i(T)^{1/2}\,\kappa_i^{1/2} 
\leq \frac{C\sqrt{\epsilon}}{\Delta_2}\,\Big(\frac{1}{n}\sum_{i=1}^n r_i(T)\Big)^{1/2} \Big(\frac{1}{n}\sum_{i=1}^n \kappa_i\Big)^{1/2},
\end{eqnarray*}
where $C<\infty$ is a constant depending only on $\|g\|_\infty$. In view of \eqref{normbound} and \eqref{contr_bound}, this readily yields the second assertion when one takes limits. 

\medskip

To deduce the third assertion, we introduce the stopping times 
\eq
\bar{\tau}_L=\inf\Big\{t\geq0: \frac{1}{n}\sum_{i=1}^n \bar{X}_i(t)\geq L\Big\},\quad L\in\nn.
\en
Then, the prelimit expression in the third assertion cannot exceed 
\[
P\Big(\sup_{0\leq s<t\leq T,t-s<\epsilon} \Big|\frac{1}{n}\sum_{i=1}^n 
\int_{s\wedge\bar{\tau}_L}^{t\wedge\bar{\tau}_L} f'(\bar{X}_i(a))\sigma({X}_i(a))\,\mathrm{d}W_i(a)\Big|>\Delta_3\Big)
+P(\bar{\tau}_L\leq T). 
\]
We now cover the interval $[0,T]$ by intervals $[0,2\epsilon],\,[\epsilon,3\epsilon],\,[2\epsilon,4\epsilon],\ldots$, use the union bound and then apply Bernstein's inequality in the form of Exercise 3.16 on page 153 in \cite{RY} (note that Problem 3.4.7 in \cite{MR1121940} can be applied to the sum of stochastic integrals in the last display) to obtain the upper bound
\[
\frac{2T}{\epsilon}\exp\Big(-c\frac{\Delta_3^2 n}{\epsilon L}\Big)+P(\bar{\tau}_L\leq T),
\]
where $c>0$ is a constant depending only on $\|f'\|_\infty$ and $\|g\|_\infty$. By taking limits we deduce that the left-hand side in the third assertion is bounded above by $P(\bar{\tau}_L\leq T)$ for every $L\in\nn$. Finally, combining Markov's inequality with \eqref{normbound} and passing to the limit $L\rightarrow\infty$, we end up with the third assertion. 

\medskip

All in all, we have shown that the set $\{\pi(\bar{\rho}^n):\;n\in\nn\}$ is precompact in $C([0,T],M_1(\rr_+))$. In addition, we observe that the identity function on $\rr_+$ belongs to $\mathcal{F}$, so that the argument above implies that the sequence $(\bar{\rho}^n(\cdot),x)$, $n\in\nn$ is tight on $C([0,T],\rr)$. Putting these two facts together we conclude that it is possible to extract a subsequence of $\pi(\bar{\rho}^n)$, $n\in\nn$, which converges in $C([0,T],M_1(\rr_+))$ and such that the corresponding functions $(\bar{\rho}^n(\cdot),x)$ converge in $C([0,T],\rr)$. By Theorem 6.9 in \cite{Vi} such a subsequence converges in $\mathcal{Y}=C([0,T],\tilde{M}_1(\rr_+))$. \ep

\section{A Case Study}

In this section we explain briefly how the stochastic differential equations in \eqref{SDE1} arise in the study of stability of the short term interbank lending system. We shall view the diffusion \eqref{SDE1} as a dynamic system that describes the monetary flows between banks from the point of view of a financial regulator. Then, taking a particular functional form of the drift and diffusion coefficients $b$, $\sigma$ and letting the number of banks go to infinity, we analyze the limiting object of Corollary \ref{LLN_cor}, viewed as the evolution of the distribution of money in a large interbank lending system. 

\medskip

Suppose that $X_i(t) \ge 0$ represents the amount of cash and liquid assets (monetary reserve) of bank $i$ at time $t$ in a financial system. In order to prevent a financial crisis in the system, a financial regulator typically requires banks to hold a sufficient amount of monetary reserves. Due to daily banking activities some banks sometimes need additional money to fulfill the requirement, while other banks have excess monetary reserves allowing them to lend money to others. The interbank lending market exists as an institution, which matches the demand and supply of such reserves allowing a bank $i$ to borrow from or lend money to other banks $j\neq i$ for a short term (for example, overnight).  

\medskip

From the perspective of a financial regulator the empirical measure $\rho^n(\cdot)= \frac{1}{n}\sum_{i=1}^{n} \delta_{X_{i}(\cdot)}$ is an important quantity to be monitored over time. When the empirical measure is close to the point mass $\delta_0$, the regulator wants to intervene in the system, whereas otherwise the system evolves free of interventions. The day-to-day transactions from one bank to another in the interbank lending market are typically quite large, which justifies the usage of macroscopic diffusions as in \eqref{SDE1} as models of the monetary reserves $(X_1(\cdot),X_2(\cdot),\ldots,X_n(\cdot))$. In \eqref{SDE1} the drift coefficient $b$, which stands for the monetary flow to or from bank $i$, is a function of the current monetary reserve $X_i(\cdot)$ and the current empirical measure $\rho^n(\cdot)$, which represents the state of the whole interbank lending system; the diffusion coefficient $\sigma$ is a function of the current monetary reserve $X_i(\cdot)$ only and stands for the volatility of the monetary reserve of bank $i$ due to the daily business of that bank. 

\medskip

For example, we may consider the following choice of $b$, $\sigma$: 
\begin{equation} \label{b-sigma}
b(x, \mu)  := {\bm \delta} + \Big ( \int_{\mathbb R_{+}} y\,\mathrm{d}\mu(y) - x \Big)\,\varphi (\mu),\quad 
\sigma (x) := 2 \sqrt{x}\,, \quad x \in \mathbb R_{+} , \, \, \mu \in \tilde{M}_{1}(\mathbb R_{+}) ,   
\end{equation}
where ${\bm \delta} $ is a positive constant and $\varphi$ is a function from $\tilde{M}_{1} (\mathbb R_{+})$ to $\mathbb R_{+}$. The constant ${\bm \delta}$ can be interpreted as the (common) growth rate of the monetary reserves and $\varphi(\rho^n(\cdot))$ as the intensity of the monetary flows between banks given the current state of the interbank lending market $\rho^n(\cdot)$. The diffusion coefficient $\sigma$ in \eqref{b-sigma} implies that the variance rate of $X_i(\cdot)$ is assumed to be proportional to the size of $X_i(\cdot)$. The resulting model reads
\begin{equation} \label{SDE3}
{\mathrm d} X_{i}(t) =  \Big( {\bm \delta} + \frac{1}{n} \sum_{j=1}^{n} \big(X_{j}(t) - X_{i}(t)\big)\,\varphi (\rho^{n}(t)) \Big)\, 
{\mathrm d} t + 2\,\sqrt{X_{i}(t)}\,{\mathrm d} W_i(t) , \quad i=1,2,\ldots,n 
\end{equation}
with $(W_1(\cdot),W_2(\cdot),\ldots,W_n(\cdot))$ being an $n$-dimensional standard Brownian motion as before. 

\begin{asmp} \label{ex_asmp}
Let $\varphi$ be a bounded Lipschitz function from $\tilde{M}_{1}(\mathbb R_{+})$ to $\rr_+$. In addition, in order to exclude trivial cases, assume that the limiting initial measure $\rho^{\infty}(0) = \lambda \in \tilde{M}_{1}(\mathbb R_{+})$ has a strictly positive first moment 
\begin{equation} \label{mlam}
m_{\lambda} :=  \int_{\mathbb R_{+}} y \,\,\mathrm{d}\lambda(y) > 0 .
\end{equation}
\end{asmp}

Note that under Assumption \ref{ex_asmp} the specification \eqref{b-sigma} satisfies Assumption \ref{LDP_asmp} (with $\Theta(x)=\sqrt{x}$). In this particular case, the sum process 
\eq
S_n(\cdot):=X_1(\cdot)+X_2(\cdot)+\ldots+X_n(\cdot)\,
\en
is a squared Bessel process of dimension $n{\bm \delta}$: 
\eq
S_n(t)=S_n(0)+n{\bm \delta}t+2\int^t_0 \sqrt{S_{n}(s)}\,{\mathrm d} \beta (s), \quad t\geq0, 
\en
where
\eq
\beta(\cdot)=\sum_{i=1}^{n} \int^{\cdot}_{0} \Big(\frac{X_{i}(t)}{S_{n}(t)}\Big)^{1/2}\,{\mathrm d} W_{i}(t) 
\en 
is a standard Brownian motion by L\'evy's characterization theorem. In particular, this allows us to conclude $\mathbb E[S_n(t)] =S_n(0) + n\delta t $, $t \ge 0$. Therefore, the computation
\begin{eqnarray*}
\mathbb E\Big[\Big( \frac{S_{n}(T)}{n} - {\bm \delta} T - m_{\lambda} \Big)^{2} \Big]
\,=\, \mathbb E \Big[ \Big( \frac{1}{n} \sum_{i=1}^{n} X_{i}(0) - m_{\lambda} \Big)^{2}\Big] 
+ \mathbb E \Big[ \Big( \frac{2}{n} \int^{T}_{0} \sqrt{S_{n}(u)}\, {\mathrm d} \beta (u)\Big)^{2} \Big] \\
+ 2 \, \mathbb E \Big[ \Big( \frac{S_{n}(0)}{n} - m_{\lambda} \Big) \Big(   \frac{2}{n} \int^{T}_{0} \sqrt{S_{n}(u)}\,{\mathrm d} \beta(u) \Big) \Big], 
\end{eqnarray*}
the convergence of the initial conditions, It\^o's isometry and the Cauchy-Schwarz inequality show that
\eq
\lim_{n\rightarrow\infty} \mathbb E\Big[\Big( \frac{S_{n}(T)}{n} - {\bm \delta} T - m_{\lambda} \Big)^{2} \Big] = 0.
\en 
Hence, by the $L^2$ version of Doob's maximum inequality for continuous martingales
\[
\forall\,\epsilon>0:\;\;\;\mathbb P \Big(\sup_{0 \le t \le T} \Big \lvert \int_{\mathbb R_{+}} x\,\mathrm{d}\rho^{n}(t)(x)  
- {\bm \delta} t - m_{\lambda} \Big \rvert > \varepsilon \Big) 
\le \frac{1}{\, \varepsilon^{2}\, } \mathbb E \Big[ \Big( \frac{S_{n}(T)}{n} - {\bm \delta} T - m_{\lambda} \Big)^{2} \Big] \, \xrightarrow[n \to \infty]{} 0\,. 
\]
Thus, the sequence of (random) functions $t\mapsto \int_{\mathbb R_+} x\,\mathrm{d}\rho^{n}(t)(x)$, $n\in\mathbb N\,$ converges in the space $C([0,T],\mathbb R_+)$ to $t\mapsto{\bm \delta} t + m_{\lambda}$ in probability in the limit $\,n \to \infty\,$. Putting this together with Theorem \ref{LDP_thm} and Corollary \ref{LLN_cor} we obtain the following.

\begin{cor} \label{TheCor}Suppose that Assumption \ref{ex_asmp} holds. Then, the $\mathcal X$-valued sequence $\rho^{n}$, $n \in \mathbb N$ of empirical measures corresponding to the particle systems in \eqref{SDE3} satisfies a large deviations principle on $\mathcal X$ with scale $n$ and a good rate function. Moreover, that sequence converges in distribution to the law of the unique strong solution of the non-local SDE 
\begin{equation} \label{limitSDE3}
{\mathrm d}X(t)=\big[{\bm \delta}+\big (m_{\lambda}+{\bm \delta}t-X(t)\big)\varphi(\mathcal L(X(t))) \big]{\mathrm d}t+2\sqrt{X(t)}\,
{\mathrm d}W(t) 
\end{equation}
with initial distribution $\lambda$, where $W(\cdot)$ is a standard Brownian motion and $\mathcal L(X(t))$ is the law of $X(t)$. Lastly, the solution of \eqref{limitSDE3} satisfies $\mathbb E[X(t)]=m_{\lambda}+{\bm \delta}t$, $t\in[0,T]$. 
\end{cor}

Note that the SDE \eqref{limitSDE3} includes as special cases the SDEs satisfied by squared radial Ornstein-Uhlenbeck (OU) processes and the SDEs satisfied by squared Bessel processes. Indeed, if one lets $\varphi$ be a constant function and takes ${\bm\delta}=0$, the solution of \eqref{limitSDE3} becomes a squared radial OU process (see \cite{MR1725406}, \cite{MR1934153} and \cite{MR1997032} for the definition and properties of the latter). On the other hand, with the choice $\varphi(\cdot)\equiv 0$ the solution of \eqref{limitSDE3} becomes a squared Bessel process of dimension ${\bm \delta}$. 

\subsection{Non-local squared radial Ornstein-Uhlenbeck processes} 

In the following, we shall consider the effect of interactions through the function $\varphi(\cdot)$ on the properties of the nonlocal SDE \eqref{limitSDE3} when ${\bm \delta}=0$. To wit,
\begin{equation} \label{limitSDE4}
{\mathrm d} X(t) = \big ( m_{\lambda}  - X(t) \big) \varphi (\mathcal L(X(t)))  {\mathrm d} t + 2 \sqrt{X(t)}\,{\mathrm d} W(t) , \quad t \ge 0   
\end{equation}
and $\mathcal L(X(0)) = \lambda$. This should be thought of as an approximation of the case when the growth rate ${\bm \delta}$ is small. We will refer to \eqref{limitSDE4} as the reduced form of \eqref{limitSDE3}. The existence and uniqueness of the strong solution to the reduced form SDE \eqref{limitSDE4} can be shown as in the previous section (see also \cite{LargeVolStab}). As for the squared radial OU process, the path properties of the process $X(\cdot)$ in \eqref{limitSDE4} are determined by the bounds of $\varphi$. 

\begin{prop} \label{Xboundary} Suppose that Assumption \ref{ex_asmp} holds and that $\lambda$ is a point mass at some $x_0>0$ (so that $X(0)=x_0$ and $m_\lambda=x_0$). Then, the solution $X(\cdot)$ of \eqref{limitSDE4} has the following properties.

\medskip 

\noindent $\bullet$ If $m_{\lambda} \inf_{\mu \in \tilde{M}_{1}}\varphi(\mu) > 2 $, then $X(\cdot)$ is transient; if $m_{\lambda} \sup_{\mu \in \tilde{M}_{1}} \varphi (\mu) \le 2$, then $X(\cdot)$ is recurrent. 

\medskip 

\noindent $\bullet$ If $m_{\lambda} \inf_{\mu \in \tilde{M}_{1}}\varphi(\mu) \ge 2$, then $X(\cdot)$ almost surely never hits the origin. 

\medskip 

\noindent $\bullet$ If $m_{\lambda} \inf_{\mu \in \tilde{M}_{1}}\varphi(\mu) > 1 $ and $m_{\lambda} \sup_{\mu \in \tilde{M}_{1}}\varphi(\mu) < 2 $, then $X(\cdot)$ hits the origin in finite time with some positive probability: 
\[
\forall\,T>0:\quad P ( X(t) = 0 \text{ for some } t \in [0, T] ) > 0.
\]
However, with probability one the semimartingale local time 
\[
L(\cdot; X, 0) \, :=\, \lim_{\varepsilon \downarrow 0} \frac{1}{2\varepsilon} \int^{\cdot}_{0} {\bf 1}_{\{0 \le X(t) < \varepsilon \}} {\mathrm d} \langle X \rangle(t) 
\]
accumulated by $X(\cdot)$ at the origin vanishes identically.

\medskip 

\noindent $\bullet$ If $m_{\lambda} \inf_{\mu \in \tilde{M}_{1}}\varphi(\mu) > 0 $ and $m_{\lambda} \sup_{\mu \in \tilde{M}_{1}} \varphi (\mu) \le 1$, the process $X(\cdot)$ does accumulate local time at the origin and the origin is instantaneously reflecting. 
\end{prop}

\begin{proof}
We shall transform the process $\,X(\cdot)\,$ to a more tractable one by introducing a suitable integrating factor. Define 
\begin{equation} \label{eq: zeta}
\zeta(\cdot):= \exp \Big( \frac{1}{2}\int^{\cdot}_{0} \varphi (\mathcal L(X(t)))\,{\mathrm d} t\Big), 
\quad \psi(\cdot) := \int^{\cdot}_{0} \zeta(t)^{2}\,{\mathrm d} t 
\quad \text{ and } \quad \xi(\cdot):= X(\cdot) \zeta(\cdot)^{2}\,. 
\end{equation}
Note that under Assumption \ref{ex_asmp} the function $\,\varphi\,$ is bounded above by some constant $C<\infty$, so that $1 \le \zeta(t) \le \exp (C t/2)$, $t \ge 0$. Hence, $\psi (\cdot)$ is strictly increasing, satisfies $t\le\psi(t)\le\int^{t}_{0}\exp(Cs)\,{\mathrm d}s$, $t \ge 0$ and has an inverse $\psi^{-1}$. In addition, by It\^o's formula 
\begin{equation} \label{xi}
\begin{split}
{\mathrm d} \xi(t)={\mathrm d} (X(t)\zeta(t)^{2}) 
&= X(t)\,{\mathrm d}(\zeta(t)^{2}) + \zeta(t)^2\,{\mathrm d}X(t) \\
& = m_{\lambda}\,\varphi(\mathcal L(\xi(t)/\zeta(t)^2))\,\zeta(t)^2\,{\mathrm d}t+2\sqrt{\xi(t)}\,\zeta(t)\,{\mathrm d} W(t).
\end{split}
\end{equation}
Equivalently, setting $\xi(0)=X(0)=x_0$, $\,{\bf D}(\cdot) = m_{\lambda} \varphi ({\mathcal L}(\xi(\psi^{-1}(\cdot))/\zeta(\psi^{-1}(\cdot))^2))\,$ and introducing another Brownian motion $\widetilde{W}(\cdot)$ with $\widetilde{W}(\psi(\cdot)) = \int^{\cdot}_{0} \zeta(t)\,{\mathrm d} W(t)$, we obtain 
\begin{equation} \label{eq: xi}
\begin{split}
\xi(\cdot) &= \xi(0) + \int^{\psi(\cdot)}_{0} m_{\lambda}\,\varphi (\mathcal L (\xi(\psi^{-1}(t))/\zeta(\psi^{-1}(t))^2))\,{\mathrm d}t 
+ \int^{\psi(\cdot)}_{0} 2 \sqrt{\xi(t)}\,{\mathrm d}\widetilde{W}(t) \\
&= \xi(0) + \int^{\psi(\cdot)}_{0} {\bf D}(t)\, {\mathrm d} t 
+ \int^{\psi(\cdot)}_{0} 2 \sqrt{\xi(t)}\, {\mathrm d} \widetilde{W}(t). 
\end{split}
\end{equation}

Comparing the latter SDE with the stochastic differential equation for a squared Bessel process of dimension $d$:
\eq
\mathrm{d}R(t)=d\,\mathrm{d}t+2\sqrt{R(t)}\,\mathrm{d}\widetilde{W}(t),
\en
we see that the process $\xi$ can be intuitively thought of as a squared Bessel process of variable dimension ${\bf D}(\cdot)$ running with respect to the clock $\psi(\cdot)$. On a formal level, the statements in the proposition can be established by combining the comparison theorem of Ikeda and Watanabe (see \cite{MR0471082}) and the corresponding properties of squared Bessel processes (see e.g. Chapter XI in \cite{RY}). For example, if $\, m_{\lambda} \sup_{\mu \in \tilde{M}_{1}} \varphi (\mu) < 2\,$, then 
\[
\,{\bf D}(\cdot) \le m_{\lambda} \sup_{\mu \in \tilde{M}_{1}} \varphi (\mu) =: {\bm \delta}_{+} < 2\,.
\]
Hence, by virtue of the comparison theorem of Ikeda and Watanabe (see \cite{MR0471082}) the process 
\[
\widetilde{\xi}(\cdot):= \xi(0) + \int^{\psi(\cdot)}_{0} {\bm \delta}_{+}\,{\mathrm d}t 
+ \int^{\psi(\cdot)}_{0} 2 \sqrt{\widetilde{\xi}(t)}\,{\mathrm d}\widetilde{W}(s) 
\]
dominates the process $\xi$ of \eqref{eq: xi}; that is, $ \xi(\cdot) \le \widetilde{\xi}(\cdot)$ with probability $1$. Therefore, since for any $T>0$ the process $\widetilde{\xi}$ reaches the origin before time $T$ with positive probability, the same is true for $\xi$. Moreover, the recurrence of $\xi$ follows from the recurrence of $\tilde{\xi}$. The other statements in the proposition can be shown in a similar fashion.
\end{proof}

Below, we shall analyze the dependence of the moments of $X(\cdot)$ on the function $\varphi$. Therefore, as a preliminary we need to establish the finiteness of such moments. 

\begin{prop} \label{Xmoments}
Suppose that Assumption \ref{ex_asmp} holds and that $\lambda$ is a point mass at some $x_0>0$ (so that $X(0)=x_0$ and $m_\lambda=x_0$). Then, the solution $X(\cdot)$ of \eqref{limitSDE4} has finite moments of all positive orders $p> 0$ and, moreover, $\mathbb E[(\sup_{0 \le s \le t} X(s))^{p}] < \infty$ for all $t$. Furthermore, the stochastic integral $\int^{\cdot}_{0} X(t)^p\,{\mathrm d}\beta(t) $ with respect to any standard Brownian motion $\beta(\cdot)$ is a martingale of class DL. 
\end{prop}

\begin{proof}
We shall use the identity \eqref{xi} from the proof of Proposition \ref{Xboundary} to estimate the quantity $\mathbb E[(\sup_{0 \le s \le t}X(s))^{p}]$ from above for any fixed $p>0$ and $t\geq0$. By Assumption \ref{ex_asmp} there exists a constant $C > 0$ such that $\sup_{\mu \in \tilde{M}_{1}} \varphi(\mu) \le C$. Recall also that we may view $\xi(\cdot)$ intuitively as a time-changed, squared Bessel process of variable dimension 
\[
{\bf D}(\cdot) = m_{\lambda} \varphi (\mathcal L(\xi(\psi^{-1}(\cdot))/\zeta(\psi^{-1}(\cdot))^2) 
\le \lceil m_{\lambda} C \rceil 
\]
running according to the clock $\psi(\cdot) \le \int_0^\cdot \exp(Ct)\,\mathrm{d}t =:\widetilde{\psi}(\cdot)$. Here, $\lceil x\rceil$ is the smallest integer not less than $x$. Applying the comparison theorem of Ikeda and Watanabe (see \cite{MR0471082}) as in the proof of Proposition \ref{Xboundary}, we conclude that the squared Bessel process $\iota$ given by the strong solution of 
\eq
\iota(\cdot) = \xi(0) + \int^{\cdot}_{0} \lceil m_{\lambda} C \rceil {\mathrm d}t
+ \int^{\cdot}_{0} 2\sqrt{\iota(t)}\,{\mathrm d} \widetilde{W}(t)
\en
satisfies
\eq
\xi(\cdot)\leq \iota(\psi(\cdot))\quad\mathrm{and}\quad \xi(\psi^{-1}(t))\leq\iota(\cdot). 
\en
In particular, 
\[
\sup_{0 \le s \le t} \xi(s) = \sup_{0 \le s\le t} \xi(\psi^{-1}(\psi(s))) 
\le \sup_{0\leq s\leq \psi(t)} \iota(s) \le \sup_{0\leq s\leq \tilde{\psi}(t)} \iota(s).  
\]
Recall now that the running maximum of a reflected Brownian motion on $\rr_+$ has finite moments of all positive orders. Moreover, the moments of the running maximum of a squared Bessel process of dimension $\lceil m_{\lambda}C\rceil$ can be bounded above by the corresponding moments of the sum of squares of running maxima of $\lceil m_{\lambda}C\rceil$ reflected Brownian motions. Since $X(\cdot) = \zeta(\cdot)^{-2} \xi(\cdot)$, it follows that 
\begin{equation} \label{moments} 
\begin{split}
\mathbb E \Big[ \big(\sup_{0 \le s \le t}X(s)\big)^{p}\Big] & = \mathbb E \Big[ \big(\sup_{0 \le s \le t} \zeta(s)^{-2}\xi(s)\big)^{p} \Big] \\
& \le \mathbb E \Big[ \big(\sup_{0\le s \le t} \xi(s)\big)^{p}\Big] \le \mathbb E \Big[ \big( \sup_{0 \le s \le \widetilde{\psi}(t)}\iota(s)\big)^{p} \Big] < \infty, \quad t \ge 0 . 
\end{split}
\end{equation}

Given a fixed $a > 0$, let us denote by $\mathcal T_{a}$ the set of stopping times $\tau$ on the underlying probability space, for which $\mathbb P(\tau \le a) =1$. A stochastic integral of the form $\int^{\cdot}_{0} (X(t))^{p}\,{\mathrm d}\beta(t)$ is a local martingale and, for every $\tau \in \mathcal T_{a}$, it holds 
\begin{eqnarray*}
\mathbb E \Big[ \Big( \int^{\tau}_{0} X(t)^{p}\,{\mathrm d} \beta(t) \Big)^{2} \Big] 
\le \mathbb E \Big[ \int^{\tau}_{0} X(t)^{2p}\,{\mathrm d} t\Big] 
\le \mathbb E \Big[ \int^{a}_{0} X(t)^{2p}\,{\mathrm d} t \Big] 
\le a\,\mathbb E [ (\sup_{0 \le t \le a} X(t))^{2p}]\\
\le a\,\mathbb E \Big[ (\sup_{0 \le t \le \widetilde{\psi}(a)} \iota(t))^{2p} \Big] < \infty \,, 
\end{eqnarray*}
where the last inequality is due to \eqref{moments}. Thus,  
\[
\sup_{\tau \in \mathcal T_{a}} \mathbb E \Big[ \Big(\int^{\tau}_{0} X(t)^{p}\,{\mathrm d} \beta (t) \Big)^{2}\Big] < \infty. 
\]
That is, the family 
\[
\Big \{ \int^{\tau}_{0} X(t)^p\,{\mathrm d}\beta(t):\;\;\tau\in\mathcal T_{a}\Big\} 
\]
is uniformly integrable. Since this is true for every $a > 0$, the local martingale $\int^{\cdot}_{0} X(t)^{p}\,{\mathrm d} \beta(t)$ is of class DL and therefore a martingale (see Problem 1.5.19 in \cite{MR1121940}) for every fixed $p>0$. 
\end{proof}

\subsection{Moments} \label{sec: var}

We shall now analyze the dependence of the moments of the solution $X(\cdot)$ of \eqref{limitSDE4} on the function $\varphi$ in the situation of Propositions \ref{Xboundary} and \ref{Xmoments}. An application of the last statement of Proposition \ref{Xmoments} yields  
\begin{eqnarray*} 
\mathbb E[X(t+h)-X(t)]=\mathbb E\Big[\int^{t+h}_{t} (m_{\lambda} - X(s))\,\varphi(\mathcal L(X(s)))\,{\mathrm d}s\Big]
+\mathbb E \Big[\int^{t+h}_{t} 2 \sqrt{X(s)}\,{\mathrm d}W(s)\Big] \\
=\int^{t+h}_{t} \big( m_{\lambda} - \mathbb E[X(s)] \big)\,\varphi(\mathcal L(X(s)))\,{\mathrm d}s 
\end{eqnarray*}
for all $t\geq0$, $h\neq0$ such that $t+h\geq0$. Dividing both sides by $h$ and taking the limit $h\rightarrow0$, one obtains an ODE for the first moment $\,\mathbb E[ X(t)]\,$: 
\[
\frac{ {\mathrm d}}{ {\mathrm d} t}\,\mathbb E[X(t)] = (m_{\lambda} - \mathbb E[ X(t)] )\,\varphi (\mathcal L(X(t))),\quad t \ge 0 
\]
with the initial value $\,\mathbb E[X(0)]=m_{\lambda}=x_{0}\,$. The solution of this ODE is 
\begin{equation}\label{eq: expecX}
\mathbb E[X(t)] = m_{\lambda} = x_{0}, \quad t\ge 0. 
\end{equation}
We conclude that one needs to consider higher moments of $X(\cdot)$ to see the effect of the interactions through the function $\varphi$.

\medskip

Applying It\^o's formula to $X^2$ we obtain 
\[
X(t+h)^2-X(t)^2=\int^{t+h}_{t} \Big [4X(s)-2\,\varphi(\mathcal L(X(s)))\big(X(s)^2-m_{\lambda}X(s)\big)\Big]\,{\mathrm d} s + \int^{t+h}_{t} 4\,X(s)^{3/2}\,{\mathrm d}W(s)   
\]
for all $t\geq0$, $h\neq0$ such that $t+h\geq0$. Taking the expectation of both sides and using the last statement in Proposition \ref{Xmoments} with $p = 3/2$, one computes  
\[
\mathbb E[X(t+h)^2]-\mathbb E[X(t)^2] = \int^{t+h}_{t} \Big[ 4 m_{\lambda} - 2\,\varphi (\mathcal L (X(s))) 
\big ( \mathbb E [X(s)^2]-m_{\lambda}^{2}\big)\Big]\,{\mathrm d} s \, . 
\]
Dividing both sides by $\,h\,$ and letting $h$ go to zero, we obtain an ODE for the variance $V(\cdot):=\text{Var}(X(\cdot)) := \mathbb E[X(\cdot)^2]-m_{\lambda}^{2}$ of $X(\cdot)$: 
\eq\label{varODE}
\frac{ {\mathrm d} V(t)}{ {\mathrm d} t}  = 4 m_{\lambda} - 2\,\varphi (\mathcal L (X(t)))V(t), \quad t \ge 0  
\en
with initial value $V(0)=0$. Thus, the variance $V(\cdot)$ increases (decreases, respectively) in time if $\varphi(\mathcal L(X(\cdot)))V(\cdot) < 2 m_{\lambda} \,$ ($ \varphi(\mathcal L(X(\cdot))) V(\cdot) > 2 m_{\lambda} \,$, respectively). In other words, when the interaction through the function $\varphi(\cdot)$ is relatively small (large, respectively), the variance increases (decreases, respectively). Since $V(0) = 0$, the variance $V(\cdot)$ increases in a neighborhood of $t = 0$. More precisely, if one defines the constant 
\[
t^{\ast}:= \inf\{t > 0: \varphi (\mathcal L (X(t))) V(t) = 2 m_{\lambda}\}, 
\] 
the variance increases on the interval $[0,t^{\ast})$. In addition, setting $\,{\bm\phi}(\cdot):=\varphi(\mathcal L (X(\cdot)))\,$ we can write the solution of the ODE \eqref{varODE} as 
\begin{equation} \label{variance}
\begin{split}
V(\cdot)&=  \Big[ \int^{\cdot}_{0} 4 m_{\lambda} \exp \Big( 2 \int^{t}_{0} {\bm \phi}(s)\,{\mathrm d}s\Big)\,{\mathrm d}t \Big] \cdot \exp \Big( - 2 \int^{\cdot}_{0} {\bm \phi}(t)\,{\mathrm d}t \Big) \\
&= 4 m_{\lambda} \int^{\cdot}_{0} \exp \Big(-2\int^{\cdot}_{t} {\bm \phi}(s)\,{\mathrm d}s \Big)\,{\mathrm d}t. 
\end{split}
\end{equation}
The higher order moments of $X(\cdot)$ can be computed in a similar manner.

\subsection{A linear first-order PDE}
Let us introduce the Laplace transform of the solution $X(\cdot)$ of \eqref{limitSDE4}:
\begin{equation} \label{Laplace}
U(t, x) := \mathbb E [\exp ( - x X(t))] = \int_{\mathbb R_{+}} e^{-x y}\,\mathcal L(X(t))({\mathrm d}y), \quad (t, x) \in (\mathbb R_{+})^{2} ,  
\end{equation}
where $\mathcal L(X(t))$ is the law of $X(t)$ as before. Note that $\mathcal L(X(\cdot))$ depends implicitly on the choice of the interaction function $\varphi:\,\tilde{M}_{1}(\mathbb R_{+}) \to \mathbb R_{+}$. 

\medskip

In a fashion similar to the arguments in section \ref{sec: var}, we may apply It\^o's formula to $\exp(-xX(t))$, integrate over the interval $(t,t+h)$, take the expectation on both sides and pass to the limit $h \downarrow 0$ to obtain
\[
\frac{\partial U}{\partial t} (t,x) 
= \big( -  m_{\lambda}\, U(t,x) + \mathbb E [ X(t)\,e^{-x X(t)}] \big)\,x\,\varphi(\mathcal L(X(t))) 
+ 2\,x^{2}\,\mathbb E[X(t)e^{-x X(t)}] \, , 
\]
$(t,x)\in(\mathbb R_+)^2$. Noting that $\partial U(t,x) / \partial x = - \mathbb E[ X(t) \exp ( - x X(t))] $ and recalling the notation ${\bm \phi}(\cdot)=\varphi(\mathcal L(X(\cdot)))$ we can rewrite the latter PDE in form of a linear first-order equation:
\eq \label{1pde}
\frac{\partial U}{\partial t}(t,x)  + x \big( {\bm \phi}(t)  + 2x \big)  \frac{\partial U}{\partial x}(t,x) 
+ m_{\lambda}\,x\,{\bm \phi}(t)\,U(t,x) = 0 , \quad (t,x) \in (\mathbb R_{+})^{2} 
\en
with initial value $U(0,x)=\int_{\mathbb R_{+}} e^{-xy}\,\lambda ({\mathrm d} y)$.  

\subsection{Stationary distribution} 
Suppose that the solution $X(\cdot)$ of \eqref{limitSDE4} is recurrent, so that it has a unique stationary distribution $\alpha$. In particular, this is the case under the condition $m_\lambda\sup_{\mu\in \tilde{M}_1(\rr)} {\bm \varphi}(\mu)\leq 2$ (see Proposition \ref{Xboundary}). In this case, it follows from \eqref{1pde} that the Laplace transform ${\mathfrak u}(x):=\int_{\mathbb R_{+}} e^{-xy}\,\alpha({\mathrm d}y) $ satisfies 
\[
({\bm \varphi}^{\ast} + 2 x) \frac{ {\mathrm d}	 {\mathfrak u}}{ {\mathrm d} x}(x) + {\bm \varphi}^{\ast}\,m_{\lambda}\,{\mathfrak u}(x) = 0 , \quad x \in \mathbb R_{+} , 
\]
where ${\bm \varphi}^{\ast} := \varphi (\alpha)$. Solving this ODE with the boundary value $\mathfrak u(0) = 1$ we obtain the Laplace transform 
\[
\mathfrak u (x) = \big(1 +  (2/{\bm \varphi}^{\ast}) x\big)^{-{\bm \varphi}^{\ast} m_{\lambda} / 2} , \quad x \in \mathbb R_{+} 
\]
of the Gamma distribution with parameters ${\mathfrak a}:=2/{\bm \varphi}^{\ast}$, ${\mathfrak b}:= {\bm \varphi}^{\ast}\,m_{\lambda}/2$:  
\begin{equation} \label{LAst}
\alpha({\mathrm d}y) = \frac{1}{ {\mathfrak a}\, \Gamma({\mathfrak b})} \Big( \frac{y}{{\mathfrak a}}\Big)^{{\mathfrak b}-1} \exp \Big( -\frac{y}{\mathfrak a} \Big) \, {\mathrm d} y, \quad y \in \mathbb R_{+}. 
\end{equation}
In addition, the above derivation shows that the function $\varphi:\,\tilde{M}_{1}(\mathbb R_{+}) \to \mathbb R_{+}$ and the stationary distribution $\alpha$ of \eqref{LAst} must be coupled through $\varphi(\alpha)={\bm \varphi}^{\ast}$.

\medskip 

\bibliographystyle{plain}
\bibliography{LargeCIRarXiv}

\begin{thebibliography}{10}

\bibitem{BP}
Richard~F. Bass and Edwin~A. Perkins.
\newblock Degenerate stochastic differential equations with {H}\"older
  continuous coefficients and super-{M}arkov chains.
\newblock {\em Trans. Amer. Math. Soc.}, 355(1):373--405 (electronic), 2003.

\bibitem{BDF}
Amarjit Budhiraja, Paul Dupuis, and Markus Fischer.
\newblock Large deviation properties of weakly interacting processes via weak
  convergence methods.
\newblock {\em Ann. Probab.}, 40(1):74--102, 2012.

\bibitem{DG}
Donald~A. Dawson and J{\"u}rgen G{\"a}rtner.
\newblock Large deviations from the {M}c{K}ean-{V}lasov limit for weakly
  interacting diffusions.
\newblock {\em Stochastics}, 20(4):247--308, 1987.

\bibitem{DSVZ}
Amir Dembo, Mykhaylo Shkolnikov, S.~R.~Srinivasa Varadhan, and Ofer Zeitouni.
\newblock Large deviations for diffusions interacting through their ranks.
\newblock {\em arxiv: 1211.5223}, 2012.

\bibitem{MR1725406}
K.~D. Elworthy, Xue-Mei Li, and M.~Yor.
\newblock The importance of strictly local martingales; applications to radial
  {O}rnstein-{U}hlenbeck processes.
\newblock {\em Probab. Theory Related Fields}, 115(3):325--355, 1999.

\bibitem{Fouque2012}
Jean-Pierre Fouque and Joseph Langsam.
\newblock {\em The Handbook of Systemic Risk}.
\newblock Cambridge University Press, 2012.

\bibitem{Ga}
J{\"u}rgen G{\"a}rtner.
\newblock On the {M}c{K}ean-{V}lasov limit for interacting diffusions.
\newblock {\em Math. Nachr.}, 137:197--248, 1988.

\bibitem{MR1997032}
Anja G{\"o}ing-Jaeschke and Marc Yor.
\newblock A survey and some generalizations of {B}essel processes.
\newblock {\em Bernoulli}, 9(2):313--349, 2003.

\bibitem{MR0471082}
Nobuyuki Ikeda and Shinzo Watanabe.
\newblock A comparison theorem for solutions of stochastic differential
  equations and its applications.
\newblock {\em Osaka J. Math.}, 14(3):619--633, 1977.

\bibitem{MR1121940}
Ioannis Karatzas and Steven~E. Shreve.
\newblock {\em Brownian motion and stochastic calculus}, volume 113 of {\em
  Graduate Texts in Mathematics}.
\newblock Springer-Verlag, New York, second edition, 1991.

\bibitem{McK}
H.~P. McKean, Jr.
\newblock Propagation of chaos for a class of non-linear parabolic equations.
\newblock In {\em Stochastic {D}ifferential {E}quations ({L}ecture {S}eries in
  {D}ifferential {E}quations, {S}ession 7, {C}atholic {U}niv., 1967)}, pages
  41--57. Air Force Office Sci. Res., Arlington, Va., 1967.

\bibitem{RY}
Daniel Revuz and Marc Yor.
\newblock {\em Continuous martingales and {B}rownian motion}, volume 293 of
  {\em Grundlehren der Mathematischen Wissenschaften [Fundamental Principles of
  Mathematical Sciences]}.
\newblock Springer-Verlag, Berlin, third edition, 1999.

\bibitem{LargeVolStab}
Mykhaylo Shkolnikov.
\newblock Large volatility-stabilized markets.
\newblock {\em arXiv:1102.3461}, 2011.

\bibitem{MR2190038}
Daniel~W. Stroock and S.~R.~Srinivasa Varadhan.
\newblock {\em Multidimensional diffusion processes}.
\newblock Classics in Mathematics. Springer-Verlag, Berlin, 2006.
\newblock Reprint of the 1997 edition.

\bibitem{Sz}
Alain-Sol Sznitman.
\newblock Topics in propagation of chaos.
\newblock In {\em \'{E}cole d'\'{E}t\'e de {P}robabilit\'es de {S}aint-{F}lour
  {XIX}---1989}, volume 1464 of {\em Lecture Notes in Math.}, pages 165--251.
  Springer, Berlin, 1991.

\bibitem{Vi}
C{\'e}dric Villani.
\newblock {\em Optimal transport}, volume 338 of {\em Grundlehren der
  Mathematischen Wissenschaften [Fundamental Principles of Mathematical
  Sciences]}.
\newblock Springer-Verlag, Berlin, 2009.
\newblock Old and new.

\bibitem{MR1934153}
Marguerite Zani.
\newblock Large deviations for squared radial {O}rnstein-{U}hlenbeck processes.
\newblock {\em Stochastic Process. Appl.}, 102(1):25--42, 2002.

\end{thebibliography}

\end{document}